\documentclass[a4paper,12pt]{amsart}
\usepackage{amsmath}
\usepackage{amsmath}
\usepackage{amssymb}
\usepackage{amscd}
\usepackage{mathrsfs}
\usepackage[all]{xy}
\usepackage[dvipdfm]{graphicx}
\def\mathcal{\mathscr}
\everymath{\displaystyle}
\newtheorem{thm}{Theorem}[section]
\newtheorem{lem}[thm]{Lemma}
\newtheorem{cor}[thm]{Corollary}

\theoremstyle{definition}

\newtheorem{rem}[thm]{Remark}

\newcommand{\mca}[1]{{\mathcal{#1}}}

\def\Z{{\mathbb Z}}

\def\R{{\mathbb R}}

\def\image{\text{\rm Im}\,}

\def\ep{\varepsilon} 
\def\ecc{\text{\rm ECC}}
\def\ech{\text{\rm ECH}}
\def\PD{\text{\rm PD}}
\def\supp{\text{\rm supp}\,}
\def\vol{\text{\rm vol}}

\begin{document}
\pagestyle{plain}
\thispagestyle{plain}

\title[Dense existence of periodic Reeb orbits and ECH spectral invariants]
{Dense existence of periodic Reeb orbits and ECH spectral invariants}

\author[Kei Irie]{Kei Irie}
\address{Research Institute for Mathematical Sciences, Kyoto University,
Kyoto 606-8502, Japan}
\email{iriek@kurims.kyoto-u.ac.jp}
\subjclass[2010]{70H12, 53D42} 

\begin{abstract}
In this paper, we prove (1): for any closed contact three-manifold with a $C^\infty$-generic contact form, 
the union of periodic Reeb orbits is dense, 
(2): for any closed surface with a $C^\infty$-generic Riemannian metric, 
the union of closed geodesics is dense. 
The key observation is $C^\infty$-closing lemma for 3D Reeb flows, 
which follows from the fact that 
the embedded contact homology (ECH) spectral invariants recover the volume. 
\end{abstract}

\maketitle

\section{Introduction}

For any contact manifold $(Y, \lambda)$, where $\lambda$ is the contact form, 
the Reeb vector field $R$ is defined by equations 
$d\lambda(R, \, \cdot \,)=0$ and 
$\lambda(R)=1$. 

The first result of this paper is that, 
for any closed contact three-manifold with a $C^\infty$-generic contact form, 
the union of periodic Reeb orbits is dense in the manifold. 
To state the result more formally, recall that 
a set $F$ in a topological space $X$ is called residual, 
if $F$ contains a countable intersection of open and dense sets in $X$. 

\begin{thm}\label{150704_1} 
For any closed contact three-manifold $(Y, \lambda)$, 
\[
\{ f \in C^\infty(Y, \R_{>0}) \mid \text{the union of periodic Reeb orbits of $(Y,f\lambda)$ is dense in $Y$} \} 
\]
is residual in $C^\infty(Y, \R_{>0})$ with respect to the $C^\infty$-topology. 
\end{thm} 

\begin{rem} 
For any closed contact manifold (not necessarily three-dimensional) with a $C^2$-generic contact from, 
the union of periodic Reeb orbits is dense; 
this is an easy consequence of the Hamiltonian $C^1$-closing lemma by Pugh-Robinson (see \cite{PR} Corollary 11.4). 
On the other hand, there exists an example of a Hamiltonian system whose $C^\infty$-small perturbations 
have no nonconstant periodic orbits (see \cite{Herman}). 
\end{rem} 

\begin{rem}\label{150828_1} 
There is a series of results 
establishing the existence of infinitely many periodic orbits for $C^\infty$-generic 
Hamiltonian/Reeb/geodesic flows. 
See \cite{GG} and the references therein. 
\end{rem} 

\begin{rem}\label{150714_1} 
In the preliminary version of this paper, we required additional assumptions in Theorem \ref{150704_1}, to directly apply Theorem 1.2 in \cite{CGHR}. 
It was pointed out by Michael Hutchings \cite{email} that we can drop these assumptions using Theorem 1.3 in \cite{CGHR} and the arguments in 
\cite{CGH} Section 2.6. 
\end{rem}

We also prove a similar result for closed geodesics (throughout this paper, all closed geodesics are assumed to be nonconstant). 
For any Riemannian metric $g$ on a manifold $\Sigma$ and $f \in C^\infty(\Sigma, \R_{>0})$, we define the metric $fg$ by 
$\| v \|_{fg}:= f(p) \| v \|_g \,(p \in \Sigma, v \in T_p\Sigma)$. 

\begin{thm}\label{150720_1} 
For any closed Riemannian surface $(\Sigma, g)$, 
\[
\{f \in C^\infty(\Sigma, \R_{>0}) \mid \text{the union of closed geodesics of $(\Sigma, fg)$ is dense in $\Sigma$} \}
\]
is residual in $C^\infty(\Sigma, \R_{>0})$ with respect to the $C^\infty$-topology. 
\end{thm} 

\begin{rem} 
As is clear from the proof of Theorem \ref{150720_1} (see Section 4), one can also prove the following variant of Theorem \ref{150720_1}: 
for any closed surface $\Sigma$, let $\mca{G}(\Sigma)$ denote the space of all $C^\infty$-Riemannian metrics on $\Sigma$. 
Then, 
\[ 
\{ g \in \mca{G}(\Sigma) \mid \text{the union of closed geodesics of $(\Sigma, g)$ is dense in $\Sigma$} \}
\]
is residual in $\mca{G}(\Sigma)$ with respect to the $C^\infty$-topology. 
\end{rem} 

To prove these theorems, we use spectral invariants in the theory of embedded contact homology (ECH). 
After some preliminaries in Section 2, we prove Theorem \ref{150704_1} in Section 3, and Theorem \ref{150720_1} in Section 4. 
The key observation  is $C^\infty$-closing lemma for 3D Reeb flows (Lemma \ref{150704_4}), 
which follows from the fact that the ECH spectral invariants recover the volume \cite{CGHR}. 

\section{Preliminaries}

In Section 2.1, we prove a preliminary result (Lemma \ref{150706_2}) on action spectra of contact manifolds. 
In Sections 2.2 and 2.3, we recall the theory of ECH, in particular quantitative aspects, very briefly. 
For precise definitions and proofs, see \cite{Hutchings} and the references therein. 

\subsection{Action spectra of contact manifolds} 

First we introduce the following notations for arbitrary set $S \subset \R$. 
\begin{itemize}
\item
For any integer $m \ge 0$, we define $S_m \subset \R$ by 
\[
S_m:= \begin{cases}  \{0\} &(m=0) \\  \{s_1 + \cdots + s_m \mid s_1, \ldots, s_m \in S\}  &(m \ge 1). \end{cases}
\]
We also define $S_+:= \bigcup_{m \ge 0} S_m$. 
\item $S$ is of class CV, if there exists a $C^\infty$-manifold $X$ and $f \in C^\infty(X)$, such that 
$S$ is contained in the set of all critical values of $f$. 
\end{itemize} 

\begin{lem}\label{150710_1} 
If $S$ is of class CV, then $S_+$ is a null (Lebesgue measure zero) set. 
\end{lem}
\begin{proof}
Suppose $S$ is contained in the set of critical values of $f \in C^\infty(X)$.
Then, for each integer $m \ge 1$, $S_m$ is contained in the set of critical values of 
$X^{\times m} \to \R; (x_1, \ldots, x_m) \mapsto f(x_1)+ \cdots + f(x_m)$, 
thus $S_m$ is a null set by Sard's theorem. 
Therefore $S_+$ is also a null set. 
\end{proof} 

For any closed contact manifold $(Y,\lambda)$, we set 
\[
\mca{P}(Y,\lambda):= \{ \gamma: \R/T_\gamma \Z \to Y \mid T_\gamma>0, \,  \dot{\gamma}=R(\gamma) \}.
\]
We define $\mca{A}: \mca{P}(Y,\lambda) \to \R_{>0}$ by 
$\mca{A}(\gamma):= T_\gamma$.
The image of this map 
$\image \mca{A} =: \mca{A}(Y,\lambda) \subset \R_{>0}$ is called the action spectrum of $(Y,\lambda)$. 

\begin{lem}\label{150706_2} 
$\mca{A}(Y,\lambda)_+$ is a closed null set in $\R_{\ge 0}$. 
\end{lem}
\begin{proof}
Let us abbreviate $\mca{A}(Y,\lambda)$ as $\mca{A}$. 
It is easy to see that $\mca{A}$ is closed, and $\min \mca{A}>0$. 
Thus $\mca{A}_+$ is closed. 
On the other hand, we can show that $\mca{A}$ is of class CV. 
The proof is similar to the case of Hamiltonian periodic orbits (see \cite{Schwarz} Lemma 3.8) and omitted. 
Then, $\mca{A}_+$ is a null set by Lemma \ref{150710_1}. 
\end{proof} 

\subsection{Embedded contact homology} 

Let $(Y,\lambda)$ be a closed contact three-manifold. 
Let $\mca{P}_0(Y,\lambda) \subset \mca{P}(Y,\lambda)$ denote the set of simple periodic Reeb orbits. 
$\mca{P}(Y, \lambda)$ and $\mca{P}_0(Y,\lambda)$ admit natural $S^1$ actions, and we denote the quotients by 
$\bar{\mca{P}}(Y,\lambda)$ and 
$\bar{\mca{P}}_0(Y,\lambda)$. 
The map $\mca{A}$ descends to $\bar{\mca{P}}(Y,\lambda)$, which we also denote by $\mca{A}$. 

Let us assume that 
the contact form $\lambda$ is nondegenerate, i.e. 
for any $\gamma \in \mca{P}(Y,\lambda)$, 
$1$ is not an eigenvalue of the linearized Poincar\'{e} map of $\gamma$. 
Under this assumption, any periodic Reeb orbit of $(Y,\lambda)$ is either elliptic or hyperbolic. 

For any $\Gamma \in H_1(Y:\Z)$, 
let $\ecc(Y,\lambda, \Gamma)$ denote the $\Z/2\,$-vector space freely generated by 
finite sets of pairs $\alpha= \{(\alpha_i, m_i)\}$ called ECH generators, such that: 
\begin{itemize}
\item The $\alpha_i$ are distinct elements in $\bar{\mca{P}}_0(Y,\lambda)$, and the $m_i$ are positive integers.
\item $\sum_i m_i [\alpha_i] = \Gamma$ in $H_1(Y: \Z)$.
\item $m_i=1$ if $\alpha_i$ is hyperbolic.
\end{itemize}
When $\Gamma=0$, the empty set $\emptyset$ is an ECH generator. 
For any $L \in \R_{>0}$, $\ecc^L(Y,\lambda, \Gamma)$ denotes the subspace of $\ecc(Y, \lambda, \Gamma)$ which is generated by 
ECH generators $\alpha = \{(\alpha_i, m_i)\}$ such that 
$\mca{A}(\alpha):= \sum_i m_i \mca{A}(\alpha_i) < L$. 
Obviously, $\mca{A}(\emptyset)=0$. 

Let $\xi:= \ker \lambda$ be the contact distribution, 
and let $d \in \Z_{\ge 0}$ denote the divisibility of $c_1(\xi)+2\PD(\Gamma)$ in $H^2(Y: \Z)$ mod torsion. 
Then, $\ecc(Y, \lambda, \Gamma)$ is relatively $\Z/d\,$-graded. 
In particular, if $c_1(\xi) + 2\PD(\Gamma)$ is a torsion in $H^2(Y:\Z)$, 
$\ecc(Y, \lambda, \Gamma)$ is relatively $\Z\,$-graded. 
In the following, we denote $\ecc$ by $\ecc_*$, to specify the relative grading. 

\begin{rem}
For any oriented two-plane field $\xi$ on a closed oriented three-manifold $Y$, 
there exists $\Gamma \in H_1(Y: \Z)$ such that $c_1(\xi) + 2\PD(\Gamma)=0$ in $H^2(Y:\Z)$. 
\end{rem} 

To define a differential on $\ecc_*(Y,\lambda, \Gamma)$, 
we fix an almost complex structure $J$ on the symplectization $Y \times \R$, which satisfies several conditions (see \cite{Hutchings} Section 1.3). 
The differential $\partial_J$, which decreases the grading by $1$, is defined by counting the number of $J$-holomorphic currents with ECH indices $1$ modulo $\R$-translation. 
It is shown that $\partial_J^2=0$, and the homology group is denoted by $\ech_*(Y,\lambda, \Gamma, J)$. 

For any ECH generators $\alpha$ and $\beta$, 
one can show 
$\langle \partial_J \alpha, \beta \rangle \ne 0 \implies \mca{A}(\alpha) \ge \mca{A}(\beta)$
using Stokes' theorem. 
Thus, for each $L \in \R_{>0}$, $\ecc^L_*(Y,\lambda, \Gamma)$ is preserved by $\partial_J$. 
We denote the homology group by $\ech^L_*(Y,\lambda, \Gamma, J)$, and 
$i^L: \ech^L_*(Y,\lambda, \Gamma, J) \to \ech_*(Y,\lambda,\Gamma, J)$ denotes the map induced by the inclusion of chain complexes. 

Given two almost complex structures $J$ and $J'$, there exists a natural isomorphism 
$\ech_*(Y,\lambda,\Gamma, J) \cong \ech_*(Y,\lambda,\Gamma, J')$. 
Moreover, for each $L \in \R_{>0}$ there exists a natural isomorphism 
$\ech^L_*(Y,\lambda,\Gamma, J) \cong \ech^L_*(Y,\lambda, \Gamma, J')$ which preserves $i^L$. 
Therefore, invariants $\ech_*(Y, \lambda, \Gamma)$, $\ech^L_*(Y, \lambda, \Gamma)$ and $i^L$ are well-defined, 
independent of almost complex structures. 

If contact forms $\lambda$ and $\lambda'$ have the same contact distribution $\xi$, namely 
$\ker \lambda = \ker \lambda' = \xi$, there exists a natural isomorphism 
$\ech_*(Y, \lambda, \Gamma) \cong \ech_*(Y, \lambda', \Gamma)$. 
We identify the RHS and the LHS, and denote it as $\ech_*(Y, \xi, \Gamma)$. 
For each $L \in \R_{>0}$, the map 
$\ech^L_*(Y,\lambda, \Gamma) \to \ech_*(Y, \xi, \Gamma)$ is well-defined and also denoted by $i^L$. 

\subsection{ECH spectral invariants} 

Let $(Y,\lambda)$ be a closed contact three-manifold, and $\xi:= \ker \lambda$. 
For any $\sigma \in \ech_*(Y, \xi, \Gamma) \setminus \{0\}$, 
let us recall from \cite{Hutchings_Quantitative} the definition of $c_\sigma(Y, \lambda) \in \R_{\ge 0}$. 
When $\lambda$ is nondegenerate, it is defined as 
\[
c_\sigma(Y,\lambda):=  \inf \{ L \mid \sigma \in \image (i^L: \ech^L(Y,\lambda,\Gamma) \to \ech(Y,\xi,\Gamma)) \}. 
\]
In the general case, i.e. $\lambda$ can be degenerate, 
let us take a sequence $(f_j)_{j \ge 1}$ in $C^\infty(Y, \R_{>0})$ such that 
$\lim_{j \to \infty} \|f_j - 1\|_{C^0} = 0$ and $f_j \lambda$ is nondegenerate for every $j \ge 1$, 
then one defines 
$c_\sigma(Y, \lambda):= \lim_{j \to \infty} c_\sigma(Y, f_j\lambda)$. 
Let us call $c_\sigma$ the ECH spectral invariant of $\sigma$. 

The spectral invariants satisfy the following properties
(see \cite{CGH} Section 2.5):

(Monotonicity) For any $f \in C^\infty(Y, \R_{\ge 1})$, $c_\sigma(Y, f\lambda) \ge c_\sigma(Y,\lambda)$.

(Scaling) For any $a \in \R_{>0}$, $c_\sigma(Y, a\lambda) = a c_\sigma(Y,\lambda)$.

(Continuity) For any sequence $(f_j)_{j \ge 1}$ in $C^\infty(Y, \R_{>0})$ such that 
$\lim_{j \to \infty} \|f_j-1\|_{C^0}=0$, $\lim_{j \to \infty} c_\sigma(Y, f_j\lambda)= c_\sigma(Y,\lambda)$.

The next lemma shows that spectral invariants are ``action selectors''. 
A special case of this lemma is proved in \cite{CGH} Lemma 3.1 (a). 

\begin{lem}\label{150715_1}
For any $\sigma \in \ech_*(Y, \xi, \Gamma) \setminus \{0\}$, 
$c_\sigma(Y,\lambda) \in \mca{A}(Y,\lambda)_+$. 
\end{lem}
\begin{proof}
First we assume that $\lambda$ is nondegenerate. 
Let us abbreviate $c_\sigma(Y,\lambda)$ by $c$, and suppose that $c \notin \mca{A}(Y, \lambda)_+$, in particular $c>0$. 
Since $\mca{A}(Y,\lambda)_+$ is closed (Lemma \ref{150706_2}), there exists $\ep \in (0, c)$ such that 
$[c-\ep, c+\ep] \cap \mca{A}(Y,\lambda)_+ = \emptyset$.
Then, $\ecc^{c-\ep}_*(Y,\lambda, \Gamma)=\ecc^{c+\ep}_*(Y,\lambda, \Gamma)$, therefore 
$\image i^{c-\ep} = \image i^{c+\ep}$. This contradicts the definition of $c$. 

Next we consider the case that $\lambda$ can be degenerate. 
Let us take a sequence $(f_j)_{j \ge 1}$ in $C^\infty(Y, \R_{>0})$ such that 
$\lim_{j \to \infty} \| f_j - 1\|_{C^1} =0$, and $f_j\lambda$ is nondegenerate for any $j$. 
($\| \cdot \|_{C^1}$ is defined by fixing local charts on $Y$.)
For each $j$, $c_\sigma(Y, f_j \lambda) \in \mca{A}(Y, f_j \lambda)_{m(j)}$ for some $m(j)$. 
Now $\sup_j m(j)<\infty$, since 
$\inf_j \min \mca{A}(Y, f_j \lambda)>0$. 
Hence, up to subsequence, we obtain 
$c_\sigma(Y,f_j\lambda)= a^1_j+ \cdots + a^m_j$, where $a^1_j, \ldots, a^m_j \in \mca{A}(Y, f_j\lambda)$, and 
$a^l_\infty:= \lim_{j \to \infty} a^l_j$ exists for any $1 \le l \le m$. 
The assumption $\lim_{j \to \infty} \| f_j-1 \|_{C^1}=0$ implies that $a^l_\infty \in \mca{A}(Y,\lambda)$ for any $l$, 
thus $c_\sigma(Y, \lambda)=a^1_\infty+ \cdots + a^m_\infty  \in \mca{A}(Y, \lambda)_m$.  
\end{proof}

Next we recall the following remarkable result from \cite{CGHR}, namely 
ECH spectral invariants recover the volume. 

\begin{thm}[Theorem 1.3 \cite{CGHR}]\label{150706_1} 
Let $(Y,\lambda)$ be any closed, connected contact three-manifold, 
$\xi = \ker \lambda$ be the contact distribution, and let $\Gamma \in H_1(Y:\Z)$. 
Suppose that $c_1(\xi) + 2\PD(\Gamma)$ is torsion in $H^2(Y:\Z)$, 
and let $I$ be an absolute $\Z$-grading of $\ech_*(Y,\xi,\Gamma)$. 
Let $(\sigma_k)_{k \ge 1}$ be a sequence of nonzero homogeneous classes in $\ech_*(Y, \xi, \Gamma)$ such that 
$\lim_{k \to \infty} I(\sigma_k)=\infty$. Then, 
\[
\lim_{k \to \infty} \frac{c_{\sigma_k}(Y,\lambda)^2}{I(\sigma_k)} = \int_Y \lambda \wedge d \lambda =: \vol(Y,\lambda).
\]
\end{thm}

\begin{cor}\label{150715_2} 
Let $\lambda$, $\lambda'$ be contact forms on a closed three-manifold $Y$ such that $\ker \lambda= \ker \lambda'= \xi$. 
Suppose that for any $\Gamma \in H_1(Y:\Z)$ such that $c_1(\xi)+ 2\PD(\Gamma)$ is torsion, and $\sigma \in \ech_*(Y, \xi, \Gamma) \setminus \{0\}$, there holds
$c_\sigma(Y, \lambda)= c_\sigma(Y, \lambda')$. 
Then, $\vol(Y,\lambda)= \vol(Y, \lambda')$. 
\end{cor}
\begin{proof}
We may assume that $Y$ is connected. 
It is known that there exists a sequence of nonzero homogeneous classes $(\sigma_k)_{k \ge 1}$ such that $I(\sigma_{k+1})= I(\sigma_k)+2$ for any $k \ge 1$. 
(This fact follows from the corresponding result in Seiberg-Witten Floer cohomology, see Corollary 2.2 in \cite{CGH}.)
Then, one can apply Theorem \ref{150706_1} to conclude $\vol(Y,\lambda)=\vol(Y,\lambda')$. 
\end{proof} 

\section{Proof of Theorem \ref{150704_1}}

The key observation is $C^\infty$-closing lemma for 3D Reeb flows (Lemma \ref{150704_4}). 
We fix local charts on $Y$, and define $\| f \|_{C^l}$ for any integer $l \ge 0$ and $f \in C^l(Y)$. 
For any $f \in C^\infty(Y)$, we set 
\[
\| f \|_{C^\infty}: = \sum_{l=0}^\infty 2^{-l} \frac{ \| f \|_{C^l}}{1 + \| f \|_{C^l}}.
\]

\begin{lem}\label{150704_4}
For any nonempty open set $U$ in $Y$ and $\ep>0$, there exists $f \in C^\infty(Y)$ such that
$\|f-1\|_{C^\infty}<\ep$ and 
there exists nondegenerate $\gamma \in \mca{P}(Y, f \lambda)$ which intersects $U$. 
\end{lem} 
\begin{proof} 
Let us take any $h \in C^\infty(Y, \R_{\ge 0})$ such that $\supp h \subset U$, $\| h\|_{C^\infty} < \ep$ and $h \not\equiv 0$. 
Obviously, $\vol(Y, (1+h)\lambda)>\vol(Y, \lambda)$. 
We are going to prove the following claim. 
\begin{quote}
\textbf{Claim:} 
There exist $t \in [0,1]$ and $\gamma \in \mca{P}(Y, (1+th)\lambda)$ which intersects $U$. 
\end{quote}
Once this claim is proved, we take $g \in C^\infty(Y)$ such that $\|g\|_{C^\infty}$ is sufficiently small, 
$g|_{\image \gamma}, dg|_{\image \gamma} \equiv 0$
(thus $\gamma \in \mca{P}(Y, e^g(1+th)\lambda)$), 
and $\gamma$ is nondegenerate as a periodic Reeb orbit with the contact form $e^g (1+th)\lambda$
(this is possible since the linearized Poincar\'{e} map of $\gamma$ is twisted by the Hessian of $g$). 
Then, $f:= e^g (1+th)$ satisfies the requirement in the lemma. 

Suppose that the claim does not hold, i.e. 
for any $t \in [0,1]$ and $\gamma \in \mca{P}(Y, (1+th)\lambda)$, 
$\gamma$ does not intersect $U$. 
Then $\mca{P}(Y, (1+th)\lambda)=\mca{P}(Y, \lambda)$ for any $t \in [0,1]$, 
since Reeb vector fields for $\lambda$ and $(1+th)\lambda$ coincide on $Y \setminus U$. 
Therefore $\mca{A}(Y, (1+th)\lambda) = \mca{A}(Y, \lambda)$ for any $t \in [0,1]$. 
Thus, for any $\Gamma \in H_1(Y:\Z)$ and $\sigma \in \ech_*(Y, \xi, \Gamma) \setminus \{0\}$, 
\[
c_\sigma(Y, (1+th)\lambda) \in \mca{A}(Y, (1+th)\lambda)_+ = \mca{A}(Y, \lambda)_+.
\]
On the other hand, $c_\sigma(Y, (1+th)\lambda)$ depends continuously on $t$. 
Since $\mca{A}(Y,\lambda)_+$ is a null set (Lemma \ref{150706_2}), 
this is a constant function of $t \in [0,1]$, thus 
$c_\sigma(Y, \lambda)= c_\sigma(Y, (1+h)\lambda)$
for any $\sigma \in \ech_*(Y, \xi, \Gamma) \setminus \{0\}$. 
Then, Corollary \ref{150715_2} shows that 
$\vol(Y,\lambda) = \vol(Y, (1+h)\lambda)$. 
This is a contradiction, thus the claim is proved. 
\end{proof} 

Let us prove Theorem \ref{150704_1}. 
For any nonempty open set $U \subset Y$, let 
\[
\mca{F}_U: = \{ f \in C^\infty (Y, \R_{>0}) \mid \text{there exists nondegenerate $\gamma \in \mca{P}(Y, f\lambda)$ which intersects $U$} \}.
\]
$\mca{F}_U$ is open in $C^\infty(Y, \R_{>0})$, and dense by Lemma \ref{150704_4}. 
Let us take a countable base $(U_i)_{i \ge 1}$ of open sets in $Y$. 
If $f \in \bigcap_{i \ge 1} \mca{F}_{U_i}$, 
then the union of periodic Reeb orbits of $(Y, f\lambda)$ is dense in $Y$. 
This completes the proof. 

\section{Proof of Theorem \ref{150720_1}}
It is enough to prove the next lemma. 

\begin{lem}\label{lem:150720} 
Let $(\Sigma, g)$ be a closed Riemannian surface. 
For any nonempty open set $U$ in $\Sigma$ and $\ep>0$, 
there exists $f \in C^\infty(\Sigma, \R_{>0})$ such that $\|f-1\|_{C^\infty} < \ep$, 
and there exists a nondegenerate closed geodesic $\gamma$ of $(\Sigma, fg)$ which intersects $U$.
\end{lem}
\begin{proof} 
Let $T^*\Sigma:= \{(q,p)\mid q \in \Sigma, p \in T_q^*\Sigma\}$ be the cotangent bundle of $\Sigma$, 
\[
\pi: T^*\Sigma \to \Sigma; \, (q,p) \mapsto q  
\]
be the canonical projection map, and $\lambda_\Sigma$ be the canonical Liouville $1$-form on $T^*\Sigma$, i.e. 
\[
\lambda_\Sigma(v):= p (\pi_*(v)) \qquad ((q,p) \in T^*\Sigma, \, v \in T_{(q,p)}T^*\Sigma). 
\]
Also, let $Y_{(\Sigma,g)}:= \{ (q,p) \in T^*\Sigma \mid \|p\|_g=1\}$. 
Then, closed geodesics of  $(\Sigma, g)$ correspond to periodic Reeb orbits of $(Y_{(\Sigma,g)}, \lambda_\Sigma)$.

Let us take $h \in C^\infty(\Sigma, \R_{\ge 0})$ such that $\|h\|_{C^\infty}< \ep$, $\supp h \subset U$ and  $h \not\equiv 0$.  
Then, $\vol(Y_{(\Sigma,g)}, \lambda_\Sigma) < \vol(Y_{(\Sigma, (1+h)g)}, \lambda_\Sigma)$. 
By same arguments as in Lemma \ref{150704_4}, we can prove that 
there exist $t \in [0,1]$ and a closed geodesic $\gamma$ of  $(\Sigma, (1+th)g)$ which intersects $U$. 
Then, at least when $\gamma$ is a simple closed geodesic, 
there exists $h' \in C^\infty(\Sigma)$ such that $\|h'\|_{C^\infty}$ is sufficiently small, 
$h'|_{\image \gamma}, dh'|_{\image \gamma} \equiv 0$, 
and $\gamma$ is nondegenerate as a closed geodesic of  $(\Sigma, (1+h')(1+th)g)$. 
Then, $f:=(1+h')(1+th)$ satisfies the requirements in the lemma. 
\end{proof} 

\begin{rem} 
The closing problem for geodesic flows was already discussed in \cite{PR} Section 10 as an open problem 
(see also \cite{Rifford} Introduction). 
The $C^1$-closing problem was solved in \cite{Rifford} Corollary 4
as a consequence of \cite{Rifford} Theorem 3, 
which claims that for any unit tangent vector $(q,v)$ on a closed Riemannian manifold (of any dimension), 
one can create a periodic orbit of the geodesic flow which passes near $(q,v)$ by a $C^1$-small conformal perturbation of the metric. 
On the other hand, our Lemma \ref{lem:150720} shows that for any point $q$ on a closed Riemannian surface, 
one can create a closed geodesic which passes near $q$ by a $C^\infty$-small conformal perturbation of the metric. 
\end{rem} 

\textbf{Acknowledgements.} 
The author appreciates Michael Hutchings for his comments on the preliminary version of this paper, in particular those concerning Remark \ref{150714_1}. 
He also appreciates Kaoru Ono for his comments and encouragements on this work, 
and an anonymous referee for many helpful suggestions. 
This work is supported by JSPS KAKENHI Grant No. 25800041.

\end{document}